\documentclass{scrartcl}

\usepackage[english]{babel}
\usepackage[utf8]{inputenc}
\usepackage[T1]{fontenc}
\usepackage{tikz}
\usepackage{verbatim}
\usepackage{dsfont}
\usepackage{amsmath}
\usepackage{amsthm}
\usepackage{amssymb}    
\usepackage{mathtools}
\usepackage{pgfplots}
\pgfplotsset{compat=newest}
\usetikzlibrary{plotmarks}
\usepackage{enumitem}
\usepackage{hyperref}

\theoremstyle{plain}

\newtheorem{theorem}{Theorem}
\theoremstyle{plain}

\newtheorem{lemma}[theorem]{Lemma}
\theoremstyle{plain}

\theoremstyle{plain}
 
\newtheorem{remark}[theorem]{Remark}
\theoremstyle{plain}

\theoremstyle{plain}

\theoremstyle{plain}

\DeclareMathOperator{\sgn}{sgn}
\newcommand{\ue}{u_\varepsilon}
\newcommand{\uh}{u_{\varepsilon,h}}

\newcommand{\R}{\mathds{R}}
\newcommand{\N}{\mathds{N}}

\newcommand{\abs}[1]{\lvert#1\rvert}

\newcommand{\norm}[1]{\lVert#1\rVert}

\begin{document}
	
\mathtoolsset{showonlyrefs=true}

\title{Finite element error estimates in $L^2$ for regularized discrete approximations to the obstacle problem
\thanks{The authors gratefully acknowledge the support by the Deutsche Forschungsgemeinschaft (DFG) through the International Research Training Group IGDK 1754 ``Optimization and Numerical Analysis for Partial Differential Equations with Nonsmooth Structures". The first author additionally acknowledges support from the graduate program TopMath of the Elite Network of the state of Bavaria and the TopMath Graduate Center of TUM Graduate School at Technische Universität M\"unchen. He is also a scholarship student of the Studienstiftung des deutschen Volkes.}}

\date{}

\author{Dominik Hafemeyer\footnotemark[2], Christian Kahle\footnotemark[2], Johannes Pfefferer\footnotemark[2]}

\maketitle

\renewcommand{\thefootnote}{\fnsymbol{footnote}}
\footnotetext[2]{Technical University of Munich, Department of Mathematics, Chair of Optimal Control, Boltzmannstra\ss e 3, 85748 Garching by Munich, Germany (\texttt Dominik.Hafemeyer@ma.tum.de, Christian.Kahle@ma.tum.de, pfefferer@ma.tum.de).}

\begin{abstract}
This work is concerned with quasi-optimal a-priori finite element error estimates for the obstacle problem in the $L^2$-norm. The discrete approximations are introduced as solutions to a finite element discretization of an accordingly regularized problem. The underlying domain is only assumed to be convex and polygonally or polyhedrally bounded such that an application of point-wise error estimates results in a rate less than two in general.
The main ingredient for proving the quasi-optimal estimates is the structural and commonly used assumption 
that the obstacle is inactive on the boundary of the domain.
Then localization techniques are used to estimate the global $L^2$-error by a local error in the inner 
part of the domain, where higher regularity for the solution can be assumed, 
and a global error for the Ritz-projection of the solution, which can be estimated by standard techniques.
We validate our results by numerical examples.

\noindent\textbf{AMS Subject Identification 2010: 35J20 35J50 65G99 65N15}
\end{abstract}

\section{Introduction}
This paper is concerned with error estimates for discrete approximations to the solution of the obstacle problem. Concerning the underlying domain, we only assume that it is 
polygonally or polyhedrally bounded and convex. Under a structural and commonly used assumption on the obstacle we show that the sequence of discrete approximations possesses a convergence rate close to two in the $L^2$-norm. Thus, we obtain convergence results similar to those of the Ritz-projection of the solution. This is the main contribution of the present paper.

Before going into further details, let us review common discretization concepts for the obstacle problem and related convergence results from the literature. 
The first approach consists of a direct discretization of the variational inequality (corresponding to the obstacle problem) based on linear finite elements.
For this approach, it is well known that the resulting sequence of discrete solutions exhibits a convergence rate of one in the $H^1$-norm if the domain is convex. The corresponding proof has already become classical in the meanwhile, see \cite{Falk1974}.
Essentially, it is based on the variational formulations of the problems (continuous and discrete) and standard interpolation error estimates.
In contrast, a universal approach for the derivation of optimal error estimates in the $L^2$-norm is unknown. 
It has even been shown in \cite{ChristofMeyer2018} that a duality argument, similar to that for the $L^2$-error of the Ritz-projection, 
cannot be established as the $H^2$-regularity of the solution in this situation is not sufficient in order to guarantee second order convergence in $L^2$.
To circumvent this issue, it is possible to consider point-wise error estimates since such estimates also imply estimates in $L^2$ due to the H\"older inequality. 
In \cite{Nitsche1977,MeyerThomas2013,Nochetto20140449} it is shown that a convergence rate of two (times $\log$-factors) can be achieved in $L^\infty$. 
This result requires sufficiently smooth data, and interior angles, that are small enough, 
in order to guarantee sufficiently smooth solutions due to the presence of corner and edge singularities. 
For instance, in two dimensional polygonal domains, it is well known that in general the interior angles must be less or equal to $\pi/2$ for the validity of those rates. 
For larger interior angles the convergence rates are reduced. In addition, based on the point-wise estimates, it is proven in \cite{MeyerThomas2013} that a convergence rate close
to two can be expected in $L^2$ if the domain is only assumed to be convex. However, this result requires an obstacle which is sufficiently smooth and, more importantly, which is inactive on the boundary. It is also crucial to note that all the point-wise estimates (and hence the $L^2$ estimate in \cite{MeyerThomas2013}) discussed so far only hold if a discrete maximum principle holds for the discrete solutions (at least this is the state of the art). For instance, this can be ensured by weakly acute finite element meshes. However, in practice, this is a serious restriction on the construction of finite element meshes, especially in the three dimensional case.

A second strategy to obtain approximations to the solution of the obstacle problem can be summarized as follows: 
First appropriately regularize the obstacle problem (for instance we use a Moreau-Yosida type relaxation) to get a semilinear partial differential equation, 
where the nonlinearity depends on the regularization parameter, and then truncate the regularized equation and discretize it by linear finite elements.
Typically, the regularization parameter is chosen dependent on the mesh parameter in order to balance both error contributions. This approach is pursued in the present paper. In case that the boundary is smooth enough and the data are regular enough, it is shown in \cite{Nochetto_1988_Sharp_Linfty_semilinear_elliptic_Free_boundaries} that by this type of discretization a convergence rate of two (times $\log$-factors) in $L^\infty$, and hence in $L^2$, can be achieved. 
Moreover, the proof can be extended to polygonal and polyhedral domains if the interior angles are small enough such that the appearing corner and edge singularities are mild enough. For larger interior angles the convergence rates in $L^\infty$ are again reduced. This is in agreement with corresponding discretization error estimates for semilinear partial differential equations, where the nonlinearity does not additionally depend on a (mesh parameter dependent) regularization parameter. In this case, a convergence rate of two can also be proven in $L^2$ if the underlying domains are only assumed to be convex. Of course, this raises the question if such a result (or at least a comparable one) is also valid for the approximations of the present discretization strategy. Typically, in order to obtain error estimates in $L^2$, a duality argument is applied. However, a straightforward application of the duality argument in the $L^2$-setting is not promising here as an inappropriate coupling between regularization parameter and mesh parameter cannot be avoided in this case.
Nevertheless, under the commonly used structural assumption that the obstacle is sufficiently smooth and inactive on the boundary, we show that a convergence rate of two (times $\log$-factors) in $L^2$ can be established in convex polygonal/polyhedral domains, which represents the main result of the paper.

Our proof heavily relies on the fact that in the continuous and discrete setting the obstacle is inactive in a non-empty strip at the boundary. This is deduced by basic point-wise estimates and the structural assumption that the obstacle is inactive on the boundary. Then, by using in a certain sense new results for locally discrete harmonic functions, we are able to split the discretization error in $L^2$ into two error terms. The first one is nothing else than an $L^2$-error for the Ritz-projection in the domain, where we can rely on standard estimates from the literature. The second error contribution represents an error in the interior of the domain, where the solution enjoys more regularity. In order to appropriately bound this term, we employ techniques from \cite{Nochetto_1988_Sharp_Linfty_semilinear_elliptic_Free_boundaries} (introduced there for global $L^\infty$-error estimates). However, we always take care on the local support of this error, which lies in the interior of the domain. A more detailed outline of the proof is given at the beginning of Section~\ref{sec:E}.

The paper is organized as follows: In Section~\ref{sec:C} we introduce the variational formulation to the obstacle problem and a Moreau-Yosida type relaxation to this problem. Moreover, we state basic properties of the corresponding solutions, such as regularity results, and we establish point-wise convergence of these solutions to each other. Some of the results are already known in a similar fashion in the literature. Nevertheless, we state the basic ideas in order to be self-contained. Moreover, and more importantly, this also enables us to ensure that those results do not depend on a smooth boundary in general, which is very often assumed in the literature. 
After having introduced the discrete problem in Section~\ref{sec:D}, we consider the $L^2$-error estimates in Section~\ref{sec:E}. There we start with giving a short roadmap for the remainder of the paper. This is followed by the observation that in all problems, which we consider, the obstacle is inactive in a strip at the boundary, which relies on point-wise estimates and the aforementioned structural assumption on the obstacle. Then we establish error estimates for locally discrete harmonic functions, which are new in a certain sense compared to the results from the literature. Using these estimates, standard results for the Ritz-projection, and duality arguments in the $L^\infty$-$L^1$-setting, we establish at the end of Section~\ref{sec:E} the main result of the paper, the discretization error estimates in $L^2$ in convex polygonal and polyhedral domains. Finally, in Section~\ref{sec:N} we state numerical examples which underline our theoretical findings.

Before closing the introduction, we emphasize that in all what follows $C$ denotes a generic positive constant which is always independent of the regularization parameter $\varepsilon$ and the mesh parameter $h$.

\section{The Continuous and the Regularized Problem}
\label{sec:C}
We start with introducing some notation which is used in the sequel of the paper. We let $\Omega \subset \R^N$, $N \in \{2,3\}$, denote an open, convex and polygonally/polyhedrally bounded domain with boundary $\partial\Omega$. The Sobolev spaces are classically denoted by $W^{k,p}(\Omega)$ with $k\in\N_0$ and $p\in [1,\infty]$. In case that $p=2$, we also use the notation $H^k(\Omega)$. The norms in these spaces are denoted by $\norm{\cdot}_{W^{k,p}(\Omega)}$ and $\norm{\cdot}_{H^{k}(\Omega)}$, respectively. In addition, $H^k_0(\Omega)$ denotes the completion of all functions in $C^\infty_0(\Omega)$ (the space of infinitely often differentiable functions with compact support in $\Omega$) with respect to $\norm{\cdot}_{H^k(\Omega)}$. Again classically, we denote the norm in $L^p(\Omega)=W^{0,p}(\Omega)$ by $\norm{\cdot}_{L^p(\Omega)}$. For the inner product in $L^2(\Omega)$ we use the notation $(\cdot,\cdot)$. The dual space to $H^1_0(\Omega)$ is denoted by $H^{-1}(\Omega)$ and we use the notation $\langle\cdot,\cdot\rangle$ to indicate the corresponding duality pairing.

Let us now formulate the problem which we are dealing with. For $f \in L^\infty(\Omega)$ and $\psi\in W^{2,\infty}(\Omega)$,
which satisfies $\psi\le0$ on $\partial\Omega$, we consider the variational
problem: Find $u \in \mathcal K_\psi:= \{v\in H^1_0(\Omega)\,|\, v \geq \psi \mbox{ a.e. in } \Omega\}$ such that
\begin{align}
  (\nabla u,\nabla (v- u)) &\geq  (f,v-u)\quad\forall v \in \mathcal K_\psi.
  \label{eq:C:VI}
\end{align}
That is, we discuss the classical obstacle problem for a function $u\in K_\psi\subset H^1_0(\Omega)$ with an obstacle $\psi \in
W^{2,\infty}(\Omega)$. By classical means it is possible to show that there exists a unique solution to this problem, 
see for instance \cite[Chap. II, Theorem 2.1]{KinderlehrerStampacchia_VariationalInequalities}. 
For the existence of a unique solution our regularity assumptions on the domain, the obstacle and the data can certainly be relaxed. Let us again emphasize that the assumptions stated above are crucial for our numerical analysis. In Section~\ref{sec:E} we even assume that the obstacle is inactive on the boundary.

Next, let us recall a well known reformulation of the obstacle problem which can be deduced by using concepts from convex analysis, see for instance \cite[Sections 1 and 2]{Barbu2010}.
Let $I_{\mathcal{K}_\psi}$ denote the indicator functional of the convex set $\mathcal{K}_\psi$. Then, the subdifferential $\partial I_{\mathcal{K}_\psi}$ of $I_{\mathcal{K}_\psi}$ at a point $u\in \mathcal{K}_\psi$ is given by
\[
\partial I_{\mathcal{K}_\psi}(u)=\left\{v\in H^{-1}(\Omega)\,\big|\, \langle v,u-w\rangle\geq0\, \forall w\in \mathcal{K}_\psi\right\}.
\]
Moreover, a function $u\in H^1_0(\Omega)$ solves the obstacle problem \eqref{eq:C:VI} if and only if there exists a Lagrange multiplier $\beta(u-\psi) \in \partial I_{\mathcal{K}_\psi}(u)$ such that
\begin{align} \label{eq:C:slackform}
(\nabla u, \nabla v) + \langle\beta(u-\psi), v\rangle = (f,v) \quad \forall v\in H^1_0(\Omega).
\end{align}
\begin{remark}
	Note that, as the solution $u$ to the obstacle problem is unique, the equation \eqref{eq:C:slackform} uniquely determines $\beta(u-\psi)\in H^{-1}(\Omega)$ by the relation
	\begin{align*}
	\beta(u-\psi)=f+\Delta u \in H^{-1}(\Omega).
	\end{align*}
\end{remark}

Next, we introduce a regularized problem, where we replace the Lagrange multiplier $\beta(u-\psi)$ by a suitable relaxation. Our approach follows that of
\cite{Nochetto_1988_Sharp_Linfty_semilinear_elliptic_Free_boundaries}. A similar one is also used in
\cite[Chap. IV, Sec. 5]{KinderlehrerStampacchia_VariationalInequalities}.

For $\varepsilon > 0$ we substitute $\beta(u-\psi)$ by the monotonically increasing, and  globally Lipschitz continuous function
\begin{align}
	\label{eq:R:betaEpsUncut}
	\beta_\varepsilon(s) :=
	\begin{cases}
		0, & \mbox{if } s\geq 0,\\
		s/\varepsilon & \mbox{if }s<0, 
	\end{cases}
\end{align}
and consider the semi-linear partial differential equation
\begin{equation}
	\label{eq:R:PDE}
	(\nabla \ue, \nabla v) + (\beta_\varepsilon(\ue-\psi),v) = (f,v) \quad \forall v \in H^1_0(\Omega)
\end{equation}
as an approximation of \eqref{eq:C:VI}. Due to the above formulated assumptions on $\beta_\varepsilon$, $\psi$ and $f$, existence of a unique solution
$\ue\in H^1_0(\Omega)\cap C(\bar\Omega)$ to \eqref{eq:R:PDE} follows for any $\varepsilon>0$ by arguments due to Browder and Minty, see for instance \cite[Theorem 4.7]{MR2583281}. We also note that this is an outer approximation or Moreau--Yosida relaxation, \cite{Glowinski1981}, of the
variational inequality \eqref{eq:C:VI}.

The following two lemmas about regularity issues for the obstacle problem and its regularized version are in the spirit of 
\cite[Chap. IV, Lemma 5.1 and Theorem 5.2]{KinderlehrerStampacchia_VariationalInequalities}.

\begin{lemma} \label{lem:R:boundedregularized}
	Let $u_\varepsilon \in H^1_0(\Omega)$ for $\varepsilon\in(0,1]$ denote the solution to \eqref{eq:R:PDE}. Then, $\beta_\varepsilon(u_\varepsilon-\psi)$ belongs to $L^\infty(\Omega)$ fulfilling
	\begin{equation}\label{eq:boundbetaeps}
		\lVert \beta_\varepsilon(u_\varepsilon-\psi) \rVert_{L^\infty(\Omega)} \leq \lVert f + \Delta \psi \rVert_{L^\infty(\Omega)}.
	\end{equation}
	Further, the solution $u_\varepsilon$ possesses the higher regularity $H^2(\Omega)\cap H^1_0(\Omega)$ and satisfies
	\[
		\norm{u_\varepsilon}_{H^2(\Omega)}\le C(\norm{f}_{L^\infty(\Omega)}+\norm{\Delta \psi}_{L^\infty(\Omega)})
	\]
	with a constant $C>0$ independent of $\varepsilon$.
\end{lemma}
\begin{proof}
		To prove the boundedness of $\beta_\varepsilon(u_\varepsilon-\psi)$ in $L^\infty(\Omega)$, one can proceed completely analogously to the proof of \cite[Chap. IV, Lemma 5.1]{KinderlehrerStampacchia_VariationalInequalities}. For the convenience of the reader and also to see the exact regularity requirements, let us quickly summarize the most essential steps of the proof. Since $u_\varepsilon \in H^1_0(\Omega)\cap C(\bar\Omega)$ and $\psi\in W^{2,\infty}(\Omega)$ with $\psi|_{\partial\Omega}\le0$ we have that $\beta_\varepsilon(u-\psi)|_{\partial\Omega}=0$ and $\beta_\varepsilon(u-\psi) \in H^1(\Omega) \cap L^\infty(\Omega)$ (according to \cite[Chap. II, Theorem A.1]{KinderlehrerStampacchia_VariationalInequalities}). Thus, we may choose $-(-\beta_\varepsilon(u_\varepsilon-\psi))^{p-1} = -|\beta_\varepsilon(u_\varepsilon-\psi)|^{p-1}$, which then belongs to $H^1_0(\Omega)\cap L^\infty(\Omega)$ as well (definitely for $p>2$), as a test function in \eqref{eq:R:PDE}. This yields
		\begin{align*}
			\lVert \beta_\varepsilon(u_\varepsilon-\psi) \rVert_{L^p(\Omega)}^p& = (\nabla (u_\varepsilon-\psi), \nabla ( (-\beta_\varepsilon(u_\varepsilon-\psi) )^{p-1} )+(f+\Delta \psi,-(-\beta_\varepsilon(u_\varepsilon-\psi))^{p-1})\\
			&=(1-p) (\nabla (u_\varepsilon-\psi),(- \beta_\varepsilon(u_\varepsilon-\psi))^{p-2} \beta_\varepsilon^\prime(u_\varepsilon-\psi) \nabla (u_\varepsilon-\psi) )\\
			&\quad-(f+\Delta \psi,\abs{\beta_\varepsilon(u_\varepsilon-\psi)}^{p-1}),
		\end{align*}
		where we used the integration by parts formula and the chain rule.		
		Then, employing the monotonicity of $\beta_\varepsilon$ together with $\beta_\varepsilon \leq 0$ and the H\"older inequality results in
		\[
			\lVert \beta_\varepsilon(u_\varepsilon-\psi) \rVert_{L^p(\Omega)}^p\le \lVert f + \Delta \psi \rVert_{L^p(\Omega)} \lVert \beta_\varepsilon(u_\varepsilon-\psi) \rVert_{L^p(\Omega)}^{p-1}.
		\]
		Finally, after having divided by $\lVert \beta_\varepsilon(u_\varepsilon-\psi) \rVert_{L^p(\Omega)}^{p-1}$, we take the limit $p\rightarrow\infty$ and obtain the desired bound for $\beta_\varepsilon(u_\varepsilon-\psi)$. As a consequence, the higher regularity can be deduced from \cite[Theorem 3.2.1.2]{Grisvard2011_EllipticProblems_nonsmoothDomains} after having sent $\beta_\varepsilon(u_\varepsilon-\psi)$ to the right hand side.
\end{proof}

\begin{lemma} \label{lem:R:multiplierboundedness}
	Let $u \in \mathcal K_\psi$ und $u_\varepsilon \in H^1_0(\Omega)$ denote the solutions to \eqref{eq:C:VI} and \eqref{eq:R:PDE}, respectively. Then, we have
	\begin{align}
	u_\varepsilon &\xrightarrow{\varepsilon \rightarrow 0} u \text{ weakly in } H^2(\Omega)\text{ and strongly in } C(\bar{\Omega}),\notag\\
	\beta_\varepsilon(u_\varepsilon-\psi) &\xrightarrow{\varepsilon \rightarrow 0} \beta(u-\psi) \text{ weakly in } L^2(\Omega),\label{eq:weakLagrange}
	\end{align}
	and
	\[
		\lVert \beta(u-\psi) \rVert_{L^\infty(\Omega)} \leq \lVert f + \Delta \psi \rVert_{L^\infty(\Omega)}.
	\]
\end{lemma}

\begin{proof}
	We proceed similar to the proof of \cite[Chap. IV, Theorem 5.2]{KinderlehrerStampacchia_VariationalInequalities}. However, we rely on the reformulation \eqref{eq:C:slackform} of the obstacle problem instead of considering the variational inequality \eqref{eq:C:VI}. Due to the uniform boundedness of $u_\varepsilon$ in $H^2(\Omega)\cap H^1_0(\Omega)$ and $\beta_\varepsilon(u_\varepsilon-\psi)\in L^2(\Omega)$ according to Lemma \ref{lem:R:boundedregularized}, we get the existence of functions $\hat u\in H^2(\Omega)\cap H^1_0(\Omega)$ and $\hat \beta\in L^2(\Omega)$ such that
	\begin{align*}
		u_\varepsilon &\xrightarrow{\varepsilon \rightarrow 0} \hat u \text{ weakly in } H^2(\Omega),\\
		\beta_\varepsilon(u_\varepsilon-\psi) &\xrightarrow{\varepsilon \rightarrow 0} \hat \beta \text{ weakly in } L^2(\Omega).
	\end{align*}
	Actually, we only get the convergence of subsequences at first. However, as the limits will be unique (the unique solution $u$ of the obstacle problem and the corresponding unique Lagrange multiplier $\beta(u-\psi)$), we will have the convergence of the whole sequences, and therefore we skip this detail in the following. Next, we show that $\hat u$ and $\hat \beta$ fulfill \eqref{eq:C:slackform}. Due to the weak convergence, we already know that
	\[
		(\nabla \hat u, \nabla v) + (\hat \beta, v) = (f,v) \quad \forall v\in H^1_0(\Omega).
	\]
	Thus, as the duality pairing between $H^{-1}(\Omega)$ and $H^1_0(\Omega)$ is compatible with the inner product in $L^2(\Omega)$, it only remains to show that $\hat \beta \in \partial I_{\mathcal{K}_\psi}(\hat u)$.
	In a first step towards this, we show that $\hat u$ belongs to $\mathcal{K}_\psi$, since then the subdifferential $\partial I_{\mathcal{K}_\psi}(\hat u)$ at $\hat u$ can be characterized as
	\[
		\partial I_{\mathcal{K}_\psi}(\hat u)=\left\{v\in H^{-1}(\Omega)\,\big|\, \langle v,\hat u-w\rangle\geq0\quad \forall w\in \mathcal{K}_\psi\right\}.
	\]
	As $H^2(\Omega)$ is compactly embedded in $C(\bar \Omega)$, we get that
	\[
		u_\varepsilon \xrightarrow{\varepsilon \rightarrow 0} \hat u \text{ strongly in } C(\bar\Omega).
	\]
	Assume next that there is a set $O\subset\Omega$ with $|O|>0$ and $\delta>0$ such that $\hat u \leq \psi - \delta$. By the strong convergence in $C(\bar\Omega)$ we have for $\varepsilon$ small enough that $u_\varepsilon \leq \psi -\delta/2$. Thus, by the Cauchy-Schwarz inequality and the definition of $\beta_\varepsilon$, we deduce
	\begin{align*}
	\lVert \beta_\varepsilon(u_\varepsilon - \psi) \rVert_{L^2(O)} \lVert u_\varepsilon -\psi\rVert_{L^2(O)} &\geq   (\beta_\varepsilon(u_\varepsilon - \psi), u_\varepsilon - \psi )_{L^2(O)}=\frac{1}{\varepsilon} \norm{u_\varepsilon-\psi}_{L^2(O)}^2 \ge |O| \frac{\delta^2}{4\varepsilon},
	\end{align*}
	which is a contradiction to the boundedness of the left hand side of this inequality (according to Lemma \ref{lem:R:boundedregularized}) if we send $\varepsilon$ to zero. As a consequence, we have shown $\hat u\in\mathcal{K}_\psi$. Now, we show that $\hat \beta$ belongs to $\partial I_{\mathcal{K}_\psi}(\hat u)$. By introducing appropriate intermediate functions, we get for any $w\in\mathcal{K}_\psi$, which implies $\beta_\varepsilon(w-\psi)=0$, that
	\begin{align*}
		\int_{\Omega}\hat \beta(\hat u-w)&=\int_\Omega(\hat \beta-\beta_\varepsilon(u_\varepsilon-\psi))(\hat u-w)+\int_{\Omega}\beta_\varepsilon(u_\varepsilon-\psi)(\hat u - u_\varepsilon)\\
		&\quad+\int_{\Omega}(\beta_\varepsilon(u_\varepsilon-\psi)-\beta_\varepsilon(w-\psi))((u_\varepsilon-\psi)-(w-\psi))\\
		&\ge\int_\Omega(\hat \beta-\beta_\varepsilon(u_\varepsilon-\psi))(\hat u-w)+\int_{\Omega}\beta_\varepsilon(u_\varepsilon-\psi)(\hat u - u_\varepsilon),
	\end{align*}
	where we used the monotonicity of $\beta_\varepsilon$ in the last step. Sending $\varepsilon$ to zero implies
	\[
		\int_{\Omega}\hat \beta(\hat u-w)\geq 0\quad \forall w\in\mathcal{K}_\psi,
	\]
	which means that $\hat{\beta}\in\partial I_{\mathcal{K}_\psi}(\hat u)$, and hence $u=\hat u$ and $\beta(u-\psi)=\hat \beta$. Finally, the estimate for the Lagrange multiplier $\beta(u-\psi)$ in $L^\infty(\Omega)$ is a direct consequence of the weak convergence \eqref{eq:weakLagrange} and the estimate \eqref{eq:boundbetaeps} due to the weakly lower semi-continuity of the norm.
\end{proof}
\begin{remark}\label{rem:character}
	Due to the convergence results of Lemma \ref{lem:R:multiplierboundedness} it is possible to further characterize $\beta(u-\psi)$. For any non-negative function $v\in C^{\infty}_0(\Omega)$ we have according to the definition of $\beta_\varepsilon$
	\[
		0\ge \lim_{\varepsilon\rightarrow0}\int_{\Omega} \beta_\varepsilon(u_\varepsilon-\psi)v=\int_{\Omega} \beta(u-\psi)v.
	\]
	Thus, by means of the fundamental lemma of variational calculus, we get $\beta(u-\psi)\leq 0$, and hence, $\beta(u-\psi)(u-\psi)\leq 0$ almost everywhere, such that the definition of $\beta_\varepsilon$ implies
	\[
		0\le \lim_{\varepsilon\rightarrow0}\int_{\Omega} \beta_\varepsilon(u_\varepsilon-\psi)(u_\varepsilon-\psi)=\int_{\Omega} \beta(u-\psi)(u-\psi)\leq 0
	\]
	or, equivalently,
	\[
		\norm{ \beta(u-\psi)(u-\psi)}_{L^1(\Omega)}=0.
	\]
	To summarize, this means in the almost everywhere sense
	\[
		\beta(u-\psi)\begin{cases}
						=0&\text{if } u-\psi>0,\\
						\le0&\text{if } u-\psi=0.
					 \end{cases}		
	\]
\end{remark}

The next theorem is concerned with the regularization error in $L^\infty(\Omega)$ and the related convergence rate. 
It basically reflects the results of \cite[Theorem 2.1]{Nochetto_1988_Sharp_Linfty_semilinear_elliptic_Free_boundaries}. Nevertheless, we recall the proof in order to ensure that
it does not depend on the smoothness of the boundary since in \cite{Nochetto_1988_Sharp_Linfty_semilinear_elliptic_Free_boundaries} a smooth boundary is assumed.

\begin{theorem}\label{thm:R:LinfError}
	Let $u \in \mathcal{K}_\psi$ and  $\ue \in H^{1}_0(\Omega)$
	denote the solutions to \eqref{eq:C:VI} and \eqref{eq:R:PDE}, respectively. Then there is the estimate
	\begin{align*}
		\|u-\ue\|_{L^\infty(\Omega)} \leq \varepsilon \|f+\Delta\psi\|_{L^\infty(\Omega)}.
	\end{align*}
\end{theorem}
\begin{proof}
	Let us abbreviate $e_\varepsilon=u-u_\varepsilon$. Having in mind the $L^2$-regularity of $\beta(u-\psi)$ from Lemma \ref{lem:R:multiplierboundedness}, we obtain from \eqref{eq:C:slackform} and \eqref{eq:R:PDE}
	\[
		(\nabla e_\varepsilon,\nabla v ) = (\beta_\varepsilon(\ue-\psi)-\beta(u-\psi),v) \quad \forall v \in H^1_0(\Omega).
	\]
	Next, we observe that the function $e_\varepsilon^{2p+1}$, where $p$ is an arbitrary positive integer, belongs to $H^1_0(\Omega)$ if $u$ and $u_\varepsilon$ belong to $H^1_0(\Omega)\cap L^\infty(\Omega)$. The $L^\infty(\Omega)$ regularity is given by Lemma~\ref{lem:R:multiplierboundedness}. Thus, we may choose it as a test function in the above variational equation. This yields employing the chain rule several times
	\begin{align*}
		(\beta_\varepsilon(\ue-\psi)-\beta(u-\psi),e_\varepsilon^{2p+1})&=(\nabla e_\varepsilon,\nabla e_\varepsilon^{2p+1})=\frac{2p+1}{(p+1)^2}\norm{\nabla e_\varepsilon^{p+1}}_{L^2(\Omega)}^2\\
		&\ge C\frac{2p+1}{(p+1)^2}\norm{e_\varepsilon^{p+1}}_{L^2(\Omega)}^2=C\frac{2p+1}{(p+1)^2}\norm{e_\varepsilon}_{L^{2(p+1)}(\Omega)}^{2(p+1)},
	\end{align*}
	where we applied the Poincar\'e inequality in between. Notice, that the constant from the Poincar\'e inequality is independent of $p$.
	We now estimate the term on the left hand side. Due to the definition of $\beta_\varepsilon$ and Remark \ref{rem:character} we notice that $\beta_\varepsilon(u-\psi)=\beta(u-\psi)=0$ almost everywhere if $u-\psi>0$. Then, due to the monotonicity of $\beta_\varepsilon$ we get
	\[
		0\ge(\beta_\varepsilon(u_\varepsilon-\psi)-\beta(u-\psi))(u-u_\varepsilon)\quad \text{a.e. in }\{x\in\Omega\,|\, u(x)-\psi(x)>0\}.
	\]
	According to the definition of $\beta_\varepsilon$ and Remark \ref{rem:character}, we also obtain
	\[
		0\ge(\beta_\varepsilon(u_\varepsilon-\psi)-\beta(u-\psi))(u-u_\varepsilon)\quad \text{a.e. in }\{x\in\Omega\,|\, u_\varepsilon(x)-\psi(x)>0\, \wedge \, u(x)=\psi(x)\}.
	\]
	Next, let us define $I=\{x\in \Omega\,|\, u_\varepsilon(x)-\psi(x)\le 0\, \wedge \, u(x)=\psi(x)\}$ and $e_\psi=\psi-u_\varepsilon \geq 0$.
	Then, combining the previous results yields
	\[
		(\beta_\varepsilon(\ue-\psi)-\beta(u-\psi),e_\varepsilon^{2p+1})\le (\beta_\varepsilon(\ue-\psi)-\beta(u-\psi),e_\psi^{2p+1})_{L^2(I)}.
	\]
	Due to the definition $\beta_\varepsilon(u_\varepsilon-\psi)$ and Remark \ref{rem:character}, this also implies
	\[
		\frac1\varepsilon\norm{e_\psi}_{L^{2(p+1)}(I)}^{2(p+1)}+(\beta_\varepsilon(\ue-\psi)-\beta(u-\psi),e_\varepsilon^{2p+1})\le (\abs{\beta(u-\psi)},e_\psi^{2p+1})_{L^2(I)}.
	\]
	By means of the H\"older and the Young inequality, we get
	\begin{align*}
		(\abs{\beta(u-\psi)},e_\psi^{2p+1})_{L^2(I)}&\le \norm{\beta(u-\psi)}_{L^{2(p+1)}(I)}\norm{e_\psi}_{L^{2(p+1)}(I)}^{2p+1}\\
		&\le\frac{1}{2(p+1)} \varepsilon^{2p+1}\norm{\beta(u-\psi)}_{L^{2(p+1)}(I)}^{2(p+1)}+\frac{2p+1}{2(p+1)}\frac{1}{\varepsilon}\norm{e_\psi}_{L^{2(p+1)}(I)}^{2(p+1)} \\
		&\le\frac{1}{2(p+1)} \varepsilon^{2p+1}\norm{\beta(u-\psi)}_{L^{2(p+1)}(I)}^{2(p+1)}+ \frac{1}{\varepsilon}\norm{e_\psi}_{L^{2(p+1)}(I)}^{2(p+1)},
	\end{align*}
	such that
	\[
		(\beta_\varepsilon(\ue-\psi)-\beta(u-\psi),e_\varepsilon^{2p+1})\le \frac{1}{2(p+1)} \varepsilon^{2p+1}\norm{\beta(u-\psi)}_{L^{2(p+1)}(I)}^{2(p+1)},
	\]
	and hence
	\begin{align*}
		\norm{e_\varepsilon}_{L^{2(p+1)}(\Omega)}&\le \left(C\frac{p+1}{2p+1}\right)^{\frac{1}{2(p+1)}}\varepsilon^{\frac{2p+1}{2(p+1)}}\norm{\beta(u-\psi)}_{L^{2(p+1)}(I)}\\
		&\le C^{\frac{1}{2(p+1)}}\varepsilon^{\frac{2p+1}{2(p+1)}}\norm{\beta(u-\psi)}_{L^{2(p+1)}(I)},
	\end{align*}
	where the constant $C$ is still independent of $p$. If we let $p$ tend to infinity, the desired result follows from Lemma \ref{lem:R:multiplierboundedness}.
\end{proof}

\begin{remark} \label{rm:regError}
    We later consider the error $\|u-\ue\|_{L^2(\Omega)}$. Nevertheless,
	Theorem \ref{thm:R:LinfError} gives an upper bound for the error due to the H\"older inequality. Even, in Section~\ref{sec:N:sharpness}, this rate is numerically  validated to be sharp.
\end{remark}

We close this section with a local regularity result for the solution of the Poisson equation, 
which is needed later in the proof of Lemma~\ref{lem:E:I:GmGh_L1_dual}.
\begin{lemma}
	\label{lem:L:local_W2p}
	Let $U \subset U_\delta \subset \Omega$ denote two connected subsets 
	with $\operatorname{dist}(\partial U,\partial U_\delta) \geq \delta$, $\delta>0$, with boundaries of class $C^{1,1}$.
	Let $\phi \in L^2(\Omega) \cap L^\infty(U_\delta)$ be given and let $z \in H^1_0(\Omega)$ denote the
	unique solution to
	\begin{align*}
		-\Delta z &= \phi \quad \mbox{ in } \Omega,\\
		z &= 0 \quad \mbox{ on } \partial\Omega.
	\end{align*}
	Then, for $p \in[2,\infty)$ there holds
	\begin{align}
		\|z\|_{W^{2,p}(U)} \leq
		Cp ( \|\phi\|_{L^p(U_\delta)} + \|\phi\|_{L^2(\Omega)} ),
		\label{eq:L:W2p_stability}
	\end{align}
	where the constant $C$ depends on $\delta$ but not on $p$.
\end{lemma}
\begin{proof}
	We follow a similar proof from 	\cite[Lemma 2.4]{LeykekhmanVexler2016_ThreeDimensionalParabolic_pointwiseControl}, i.e., we apply a boot strapping argument.
	First we introduce an intermediate smooth domain $U_{\delta/2}$ such that $U \subset U_{\delta/2} \subset U_\delta$ with $\mbox{dist}(\partial U, \partial U_{\delta/2}) \geq \delta/2$ and $\mbox{dist}(\partial U_{\delta/2}, \partial U_\delta) \geq \delta/2$. 
	In a first step we show $W^{1,\infty}(U_{\delta/2})$-regularity for $z$. 
	Let $\omega \in C^{\infty}(\Omega)$
	denote a smooth cut-off function on $U_{\delta/2}$, such that
	$\omega|_{U_{\delta/2}} \equiv 1$,
	$\omega|_{\Omega \setminus U_{\delta}} \equiv 0$, and
	$|\omega|_{W^{r,\infty}(\Omega)} \leq C \delta^{-r}$ for $r \in \{0,1,2\}$,
see \cite[Theorem 1.4.1 and Equation (1.42)]{Hoermander2003} for the existence of such a cut-off function and the corresponding estimates.	
We set $v := \omega z$. Then $v\in H^1_0(U_\delta)$ is the weak solution to
\begin{alignat*}{2}
	-\Delta v 
	&= \phi \omega+ (-\Delta \omega) z - 2\nabla \omega \cdot \nabla z
	=: g &\quad& \mbox{in }  U_\delta,\\
	v &= 0  && \mbox{on } \partial U_\delta.
\end{alignat*}
Due to the smoothness properties of $\omega$, the right hand side $g$ can be bounded by 
\begin{align*}
	\|g\|_{L^6(U_\delta)} \leq C \left(
	\|\phi\|_{L^6(U_\delta)} +  \|z\|_{W^{1,6}(U_\delta)} \right),
\end{align*}
where the constant $C$ depends on $\delta$. Moreover, due to the $H^2$-regularity of $z$,  as $\Omega$ is convex, we have
\begin{align*}
	\|z\|_{W^{1,6}(\Omega)}
	\leq C\|z\|_{H^2(\Omega)}
	\leq C \|\phi\|_{L^2(\Omega)}.
\end{align*}
Consequently, by elliptic regularity, c.f.  \cite[Theorem 9.9]{GilbargTrudinger}, we obtain
\begin{align*}
	\|v\|_{W^{2,6}(U_\delta)}
	\leq
	C\|g\|_{L^6(U_\delta)}
	\leq
	C( \|\phi\|_{L^6(U_\delta)} + \|\phi\|_{L^2(\Omega)} ).
\end{align*}
Since $\omega \equiv 1$ on $U_{\delta/2}$ we have $v|_{U_{\delta/2}} \equiv z|_{U_{\delta/2}}$ and therefore
\begin{align} \label{eq:L:interiorW1infty}
\|z\|_{W^{1,\infty}(U_{\delta/2})} \leq C \|z\|_{W^{2,6}(U_{\delta/2})} \leq
	C ( \|\phi\|_{L^6(U_{\delta})} + \|\phi\|_{L^2(\Omega)} ).
\end{align}
Next, we repeat the above argumentation for $U$ and $U_{\delta/2}$ with correspondingly changed cut-off function $\omega$ and auxiliary problem for $v$.
Let $p>6$ and $\omega$ denote a cut-off function such that 
$\omega \equiv 1$ on $U$ and $\omega \equiv 0$ on $\Omega \setminus U_{\delta/2}$.
As above, we obtain 
\begin{align*}
\|g\|_{L^p(U_{\delta/2})} \leq C \left( \|\phi\|_{L^p(U_{\delta/2})} + \|z\|_{W^{1,p}(U_{\delta/2})} \right)\leq C\left( \|\phi\|_{L^p(U_\delta)} + \|\phi\|_{L^2(\Omega)}\right),
\end{align*}
where we used \eqref{eq:L:interiorW1infty}. Finally,
from elliptic regularity, c.f. \cite[Theorem 9.9]{GilbargTrudinger}, we get for $p\in[2,\infty)$ the desired result,
\begin{align*}
	\|z\|_{W^{2,p}(U)} \leq
	Cp( \|\phi\|_{L^p(U_{\delta})} + \|\phi\|_{L^2(\Omega)} ),
\end{align*}
where we notice that the constant $C$ is independent of $p$. This can be seen from the proof of \cite[Theorem 9.9]{GilbargTrudinger}.
\end{proof}

\section{The Discrete Problem}
\label{sec:D}

In the following we derive optimal a-priori error estimates in $L^2(\Omega)$ for a numerical approximation
to \eqref{eq:C:VI} which is based on the regularized problem \eqref{eq:R:PDE}. We rely on the approach of
\cite{Nochetto_1988_Sharp_Linfty_semilinear_elliptic_Free_boundaries}. However, we again notice that the results from that reference are not directly applicable in our setting as in \cite{Nochetto_1988_Sharp_Linfty_semilinear_elliptic_Free_boundaries} global
$W^{2,p}$-regularity is required with arbitrarily large $p < \infty$.

Let us now introduce the numerical approximation which we are dealing with. We do not discretize \eqref{eq:R:PDE} directly but an equivalent reformulation of it. According to \eqref{eq:boundbetaeps}, we may truncate the nonlinearity $\beta_\varepsilon$ without changing the solution to \eqref{eq:R:PDE}. More precisely, if we choose the constant 
\begin{equation}\label{eq:lambda}
	\lambda := c\lVert f + \Delta\psi\rVert_{L^\infty(\Omega)}
\end{equation}
with $c\ge 1$, we may redefine $\beta_\varepsilon$ by the bounded, monotonically increasing, and globally Lipschitz continuous function
\begin{align}
	\label{eq:R:betaEps}
	\beta_\varepsilon(s) :=
	\begin{cases}
		0, & \mbox{if } s\geq 0,\\
		\max(s/\varepsilon,-\lambda), & \mbox{if } s<0,
	\end{cases}
\end{align}
without changing the solution of \eqref{eq:R:PDE}. This problem with the redefined nonlinearity is now being discretized by piecewise continuous and linear finite elements. Let $\{\mathcal T_h\}$ be a family of conforming and quasi-uniform triangulations of $\Omega$ which are admissible in the sense of Ciarlet. We denote by $h:=\max_{T\in \mathcal{T}_h} \operatorname{diam}T$ the global mesh parameter and assume that $h<1/2$. For each element $T\in\mathcal{T}_h$ we assume that it is isoparametrically equivalent either to the unit cube or to the unit simplex in $\R^N$.
On $\mathcal T_h$ we define
\begin{align*}
  V_h := \{ v_h \in C(\overline \Omega) \,|\,
  v|_{T} \mbox{ is affine } \forall T \in \mathcal T_h, \, v|_{\partial\Omega} \equiv 0\},
\end{align*}
and determine approximations to the solution $u_\varepsilon$ of \eqref{eq:R:PDE} by solving the problem: Find $u_{\varepsilon,h}\in V_h$ such that
\begin{align}
  (\nabla \uh,\nabla v_h) + (\beta_\varepsilon(\uh-\psi),v_h) = (f,v_h) \quad \forall v_h \in V_h.
  \label{eq:D:PDE}
\end{align}
For each mesh parameter $h$ the existence of a unique solution to this finite dimensional problem follows by standard arguments. 
For later reference, we define $I_h:C(\overline\Omega) \to V_h$ as the usual Lagrangian interpolation operator, and the Ritz projection of $w \in H^1_0(\Omega)$ as the function $R_h w$ in $V_h$ which satisfies
\begin{equation}\label{eq:Ritz}
  (\nabla (R_h w-w),\nabla v_h)=0 \quad \forall v_h \in V_h.
\end{equation} 
Finally, let us stress that we assume exact integration for the non-smooth nonlinearity $\beta_\varepsilon(\uh-\psi)$. 
We refer to \cite{1994_Nochetto_numericalIntegration_Obstacle},
where a lumping technique is used for the numerical approximation of the non-linear term. 

\section{Error Estimates in \texorpdfstring{$L^2(\Omega)$}{Lg}}
\label{sec:E}
In Theorem \ref{thm:R:LinfError} we have already seen that the regularization error can appropriately be bounded in $L^2(\Omega)$, even  in $L^\infty(\Omega)$. 
In the following we
derive a priori bounds with respect to the discretization parameter $h$ for the discretization error $\ue-\uh$ in $L^2(\Omega)$, see Theorem \ref{thm:E:I:ue_m_uh_dual}. 
Afterwards, we combine these results in Theorem \ref{thm:E:totalerror}.

Before going into detail, let us quickly elucidate the structure of the main part of this section, the proof of estimates for the discretization error.
Based on the assumption that $\psi<0$ on the boundary, i.e., the obstacle is inactive on the boundary, we show in a first step that there exists a (non-empty) strip $D_d$ at the boundary $\partial\Omega$ of width $d$ (independent of $\varepsilon$ and $h$) such that
\begin{equation}\label{eq:betaneighborhood}	\beta(u-\psi)=\beta_\varepsilon(u_\varepsilon-\psi)=\beta_\varepsilon(u_{\varepsilon,h}-\psi)=0\quad \text{a.e. in }D_d\subset\Omega,
\end{equation}
see Lemma \ref{lem:E:P:InactiveAtBoundary}, and hence, the constraint is inactive in the neighborhood $D_d$ of the boundary for each problem. 
The proof requires that $\varepsilon$ and $h$ are small enough as it relies on the fact that we already have 
point-wise convergence of $u_\varepsilon$ towards $u$, see Theorem \ref{thm:R:LinfError}, 
and point-wise convergence of $u_{\varepsilon,h}$ towards $u_\varepsilon$ with some (maybe not optimal) rate, see Lemma~\ref{lem:E:P:subOpt_Linf}.
According to \eqref{eq:betaneighborhood}, we also have that $R_hu_\varepsilon-u_{\varepsilon,h}$ is discretely harmonic, see \eqref{eq:L:discHarmon_1}, on $D_d$. 
This implies that there exists another strip
$D$ at the boundary (for instance of width $d/2$) such that
\begin{equation}\label{eq:locdiscreteharmonic}
	\norm{R_hu_\varepsilon-u_{\varepsilon,h}}_{H^1(D)}\le C \norm{R_hu_\varepsilon-u_{\varepsilon,h}}_{L^2(D_d\backslash D)}\le C \norm{R_hu_\varepsilon-u_{\varepsilon,h}}_{L^2(\Omega\backslash D)},
\end{equation}
where the constant $C$ depends on the distance between $D$ and $D_d$, see Theorem \ref{thm:E:D:L2D_to_L2Dd}. 
Based on this, we get after having introduced $R_hu_\varepsilon$ as an intermediate function
\begin{align}
	\norm{u_\varepsilon-u_{\varepsilon,h}}_{L^2(\Omega)}
	&\le\norm{u_\varepsilon-R_hu_\varepsilon}_{L^2(D)}+\norm{R_hu_\varepsilon-u_{\varepsilon,h}}_{L^2(D)}+\norm{u_\varepsilon-u_{\varepsilon,h}}_{L^2(\Omega\backslash D)}\notag\\
	&\le\norm{u_\varepsilon-R_hu_\varepsilon}_{L^2(D)}+C\norm{R_hu_\varepsilon-u_{\varepsilon,h}}_{L^2(\Omega\backslash D)}+\norm{u_\varepsilon-u_{\varepsilon,h}}_{L^2(\Omega\backslash D)}\notag\\
	&\le C\left(\norm{u_\varepsilon-R_hu_\varepsilon}_{L^2(\Omega)}+\norm{u_\varepsilon-u_{\varepsilon,h}}_{L^2(\Omega\backslash D)}\right),\label{eq:assembling}
\end{align}
where we introduced $u_\varepsilon$ as an intermediate function in the last step. 
Estimating the error of the Ritz-projection $R_h u_\varepsilon$ is standard, taking into account 
the $H^2(\Omega)$-regularity of $u_\varepsilon$ according to Lemma \ref{lem:R:multiplierboundedness}. 
It remains to bound the second term in the previous inequality. More precisely, we estimate the difference $u_\varepsilon-u_{\varepsilon,h}$ in $L^\infty(\Omega\backslash D)$. 
Here we rely on a duality argument as in \cite{Nochetto_1988_Sharp_Linfty_semilinear_elliptic_Free_boundaries}, see Theorem~\ref{thm:E:I:ue_m_uh_dual}. 
However, we always take care on the fact that this term only lives in the interior of the domain, 
where we have higher regularity. 
This is the main reason for having second order convergence (times a $\log$-factor) in $L^2(\Omega)$ in case of general convex polygonal/polyhedral domains.

We start with providing an $L^\infty(\Omega)$-estimate for the discretization error, which is valid in convex domains, but only has a lower convergence rate.
\begin{lemma}
\label{lem:E:P:subOpt_Linf} Let $u_\varepsilon$ and $u_{\varepsilon,h}$ be the solutions of \eqref{eq:R:PDE} and \eqref{eq:D:PDE}, respectively. Then, there is the estimate
\[\| \ue - \uh \|_{L^\infty(\Omega)} \leq C h^{2-\frac{N}{2}}( \|f \|_{L^\infty(\Omega)} + \|\Delta\psi \|_{L^\infty(\Omega)} ),\]
where the constant $C$ is independent of
$\varepsilon$ and $h$.
\end{lemma}
\begin{proof}
This follows from 
\cite[Lemma 2.2 and Theorem 2.3]{Nochetto_1988_Sharp_Linfty_semilinear_elliptic_Free_boundaries} 
using only $H^2(\Omega)$-regularity which holds in general convex domains. Equivalently, one can set $D=\emptyset$ within the proof of Theorem \ref{thm:E:I:ue_m_uh_dual} and Lemma \ref{lem:E:I:GmGh_L1_dual}. Then, by taking into account only the $H^2(\Omega)$-regularity of $z$ within the proof of Lemma \ref{lem:E:I:GmGh_L1_dual} for estimating $\|z-R_hz\|_{L^\infty(\Omega)}$, one obtains the desired result as well.
\end{proof}
Based on the previous lemma, we next show that \eqref{eq:betaneighborhood} holds.
\begin{lemma}
\label{lem:E:P:InactiveAtBoundary}
Let $u$, $u_\varepsilon$ and $u_{\varepsilon,h}$ be the solutions of \eqref{eq:C:slackform}, \eqref{eq:R:PDE} and \eqref{eq:D:PDE}, respectively. In addition assume that $\psi<0$ on the boundary.
Then, there exist constants $d>0$, $\varepsilon_0>0$ and $h_0>0$ such that for all $\varepsilon\leq \varepsilon_0$ and $h\le h_0$ there holds
 \[
 	\beta(u-\psi)=\beta_\varepsilon(u_\varepsilon-\psi)=\beta_\varepsilon(u_{\varepsilon,h}-\psi)=0\quad \text{a.e. in }D_d:=\{ x \in \Omega \,|\, \operatorname{dist}(x,\partial\Omega) \le d \}.
 \]
\end{lemma}
\begin{proof}
As the obstacle $\psi$ is a continuous function on the boundary, which represents a compact set, we obtain that there exists a $\tau>0$ such that $\psi\le -\tau$ on the boundary, and hence, there holds
$u - \psi \geq \tau$ on the boundary. Next, we notice that $u-\psi$ is a continuous function up to the boundary, see Lemma~\ref{lem:R:multiplierboundedness}. Consequently, there is a constant $d>0$ such that 
$u - \psi\geq \frac{1}{2}\tau$ on $D_d$. Further, from Theorem~\ref{thm:R:LinfError} we have that $|u(x)-\ue(x)| \leq \varepsilon\|f+\Delta \psi\|_{L^\infty}$ for all $x\in\Omega$. Consequently, there exists a constant $\varepsilon_0>0$ such that $\ue - \psi \geq \frac{1}{4}\tau$ on $D_d$ for all $\varepsilon\le \varepsilon_0$. In the same manner, now using the $L^\infty(\Omega)$-estimate from Lemma~\ref{lem:E:P:subOpt_Linf} (note that the constant there is independent of $\varepsilon$ and $h$), we deduce the existence of a constant $h_0>0$ such that for all $h\le h_0$ there holds $\uh - \psi \geq \frac{1}{8}\tau$ on $D_d$. The assertion now follows from the discussion in Remark \ref{rem:character} and the definition of $\beta_\varepsilon$ in~\eqref{eq:R:betaEps}.
\end{proof}

Next, we are concerned with proving \eqref{eq:locdiscreteharmonic}. For that reason, let us first introduce the notion of locally discrete harmonic functions as it is used in the following. Let $U_\delta$ denote a subset of $\Omega$. We call a function $w_h \in V_h$ discretely harmonic on $U_\delta$ if
\begin{align}
\label{eq:L:discHarmon_1}
(\nabla w_h,\nabla v_h) = 0 \quad \forall v_h \in V_h\cap \{v\in H^1(\Omega)\,|\, v=0\text{ a.e. in } \Omega\backslash U_\delta\}.
\end{align}
It is well known that discretely harmonic functions fulfill the following Caccioppoli-type estimate: Let $U$ and $U_\delta$ be subsets of $\Omega$ such that $U\subset U_\delta$ and $\operatorname{dist}(U,\partial U_\delta\backslash \partial\Omega)=\delta$ with $\delta>0$. Further, assume that $w_h\in V_h$ is discretely harmonic on $U_\delta$. Then for $h$ small enough (depending on $\delta$) there is the estimate
\begin{equation}\label{eq:discreteharm}
	\norm{\nabla w_h}_{L^2(U)}\le C \delta^{-1} \norm{w_h}_{L^2(U_\delta)},
\end{equation}
where the constant $C$ is independent of $\delta$. Estimates of this kind are essential when proving local energy norm estimates, which can be traced back to \cite{NitscheSchatz1974}
We also mention \cite{Demlow2011} where in contrast to \cite{NitscheSchatz1974} the assumption on quasi-uniform meshes is avoided and sharply varying grids are admitted. A more sophisticated discussion on local estimates and a survey on related results from the literature can be found in \cite{Demlow2011} as well.

In \eqref{eq:discreteharm} the norm on the right hand side is defined on $U_\delta$ but not on $U_\delta\backslash U$ as it is required for our purposes. 
However, the results from the literature can be extended to this by minor modifications. We summarize this in the following lemma.
We assume that the mesh is quasi-uniform, and only notice that the results also extend to the more general setting of \cite{Demlow2011}.

\begin{lemma}
\label{lem:L:H1D_L2Dd}
Let $U$ and $U_\delta$ be subsets of $\Omega$ such that $U\subset U_\delta$ and $\operatorname{dist}(U,\partial U_\delta\backslash \partial\Omega)=\delta$ with $\delta>0$. Further, assume that $w_h\in V_h$ is discretely harmonic on $U_\delta$ in the sense of \eqref{eq:L:discHarmon_1}. Then, there exists a constant $h_\delta>0$ (depending on $\delta$) such that for $h\leq h_\delta$ there is the estimate
\[
\norm{\nabla w_h}_{L^2(U)}\le C \delta^{-1} \norm{w_h}_{L^2(U_\delta\backslash U)},
\]
where the constant $C$ is independent of $\delta$.
\end{lemma}
\begin{proof}
For $i=1,\ldots ,4$, let $U_{i\delta/5}$ be a subset of $\Omega$ such that $U\subset U_{i\delta/5}\subset U_\delta$ and $\operatorname{dist}(U,\partial U_{i\delta/5}\backslash \partial\Omega)=i\delta/5$. Moreover,
we define the smooth cut-off function $\omega \in C^\infty(\Omega)$ which satisfies
\begin{align*}
  \omega|_{U_{2\delta/5}} \equiv 1, 
  \quad
  \omega|_{\Omega \setminus U_{3\delta/5}} \equiv 0,
  \quad 
  \mbox{and}\quad 
  |\omega|_{W^{r,\infty}(\Omega)} \leq C\delta^{-r} \text{ for } 0\le r \le 2,
\end{align*}
see \cite[Theorem 1.4.1 and Equation (1.42)]{Hoermander2003} for the existence of such a cut-off function and the corresponding estimates. By simple calculations we deduce
\begin{align}
  \|\nabla w_h \|_{L^2(U)}^2 &\le\|\omega\nabla w_h\|^2_{L^2(U_\delta)} = \int_{U_\delta} \omega^2 \nabla w_h \cdot \nabla w_h \notag\\
  &= \int_{U_{\delta}}\nabla w_h\cdot\nabla (\omega^2 w_h)
  - \int_{U_{\delta}}w_h\nabla w_h\cdot\nabla \omega^2.\label{eq:startlocal}
\end{align}
For the second term we obtain by the Cauchy-Schwarz inequality and the properties of~$\omega$
\begin{align}
  \bigg|\int_{U_\delta}&w_h\nabla w_h\cdot\nabla \omega^2\bigg|
  =2\left|\int_{U_\delta}\omega\nabla w_h\cdot w_h\nabla \omega\right|
  \le 2\|\omega\nabla w_h\|_{L^2(U_\delta)} \| w_h\nabla \omega\|_{L^2(U_\delta\setminus U)}\notag\\
  &\le C\delta^{-1}\|\omega\nabla w_h\|_{L^2(U_\delta)} \| w_h\|_{L^2(U_\delta\setminus U)}
  \le \frac{1}{4}\|\omega\nabla w_h\|_{L^2(U_\delta)}^2
  + C \delta^{-2}\| w_h\|^2_{L^2(U_\delta\setminus U)},\label{eq:local1}
\end{align}
where we applied Young's inequality in the last step. Next, we consider the first term in \eqref{eq:startlocal}.
We notice that there exists a constant $h_\delta>0$ such that for all $h\le h_\delta$ there holds 
$I_h(\omega^2w_h)\in V_h\cap \{v\in H^1(\Omega)\,|\, v=0\text{ a.e. in } \Omega\backslash U_{4\delta/5}\}$ and $I_h(\omega^2w_h) \equiv \omega^2w_h$ on $U_{\delta/5}$. 
Thus, using \eqref{eq:L:discHarmon_1} to insert $I_h(\omega^2w_h)$ we obtain
\begin{align*}
  \int_{U_\delta}\nabla w_h\cdot\nabla (\omega^2 w_h)
  &= \int_{U_\delta}\nabla w_h\cdot
  \nabla \left(\omega^2 w_h - I_h(\omega^2 w_h)\right)\\
  &= \int_{U_{4\delta/5}\setminus U_{\delta/5}}\nabla w_h\cdot
  \nabla \left(\omega^2 w_h - I_h(\omega^2 w_h)\right)\\
  &\leq \sum_{T\subset U_\delta\setminus U} \|\nabla w_h\|_{L^2(T)} \|\nabla\left(\omega^2 w_h - I_h(\omega^2 w_h)\right)\|_{L^2(T)}.
\end{align*}
For each element $T\subset U_\delta\setminus U$ we deduce by means of an inverse inequality and a standard interpolation error estimate
\begin{align*}
	\norm{\nabla w_h}_{L^2(T)} \norm{\nabla\left(\omega^2 w_h - I_h(\omega^2 w_h)\right)}_{L^2(T)}
	&\le C h^{-1}\|w_h\|_{L^2(T)} h |\omega^2 w_h|_{H^2(T)}\\
	&=C\|w_h\|_{L^2(T)} |\omega^2 w_h|_{H^2(T)}.
\end{align*}
Moreover, using the bounds for $\omega$ and its derivatives, we get by elementary calculations
\begin{align*}
  |\omega^2 w_h|_{H^2(T)}
  &\leq C \left( |\omega|_{W^{1,\infty}(T)}\|\omega\nabla w_h\|_{L^2(T)}
  + |\omega^2|_{W^{2,\infty}(T)}\|w_h\|_{L^2(T)}\right)\\
  &\leq C \left(
  \delta^{-1}\|\omega\nabla w_h\|_{L^2(T)} + \delta^{-2}\|w_h\|_{L^2(T)}
  \right).
\end{align*}
The previous inequalities imply
\begin{align}
  \int_{U_\delta}\nabla w_h\cdot\nabla (\omega^2 w_h)& \leq
  C\sum_{T\subset U_\delta\setminus U}
  \left(
  \delta^{-1}\|w_h\|_{L^2(T)}\|w\nabla w_h\|_{L^2(T)}
  + \delta^{-2}\|w_h\|^2_{L^2(T)}
  \right) \notag\\
  &\le \frac{1}{4}\|\omega \nabla w_h\|_{L^2(U_\delta)}^2
  + C\delta^{-2}\|w_h\|^2_{L^2(U_\delta \setminus U)},\label{eq:local2}
\end{align}
where we applied Young's inequality in the last step. We finally get the assertion from \eqref{eq:startlocal}, \eqref{eq:local1} and \eqref{eq:local2}.
\end{proof}

We now combine the previous results to deduce \eqref{eq:locdiscreteharmonic}.

\begin{theorem} \label{thm:E:D:L2D_to_L2Dd}
Let $u_\varepsilon$ and $u_{\varepsilon,h}$ be the solutions of \eqref{eq:R:PDE} and \eqref{eq:D:PDE}, respectively. In addition assume that $\psi<0$ on the boundary. Then, there exist a non-empty strip $D$ at the boundary and constants $\varepsilon_1>0$ and $h_1>0$ such that for all $\varepsilon\leq \varepsilon_1$ and $h\le h_1$ there holds
\[
	\norm{R_hu_\varepsilon-u_{\varepsilon,h}}_{H^1(D)}\le C \norm{R_hu_\varepsilon-u_{\varepsilon,h}}_{L^2(\Omega\backslash D)}.
\]
\end{theorem}
\begin{proof}
Define $D:=\{ x \in \Omega \,|\, \mbox{dist}(x,\partial\Omega) \le d/2 \}$, where $d$ denotes the width of the strip $D_d$ in Lemma \ref{lem:E:P:InactiveAtBoundary}.
 From the same lemma we obtain that $R_hu_\varepsilon-u_{\varepsilon,h}$ is discretely harmonic on $D_d$ for all $h\le h_0$ and $\varepsilon\le \varepsilon_1:=\varepsilon_0$ as
 $\beta_\varepsilon(u_\varepsilon-\psi)=\beta_\varepsilon(u_{\varepsilon,h}-\psi)=0$ on $D_d$. Consequently, employing Lemma \ref{lem:L:H1D_L2Dd} there exists a constant $h_d$ such that for all $h\le h_1:=\min\{h_0,h_d\}$ there holds
\[
	\norm{\nabla(R_hu_\varepsilon-u_{\varepsilon,h})}_{L^2(D)}
	\le C \norm{R_hu_\varepsilon-u_{\varepsilon,h}}_{L^2(D_d\setminus D)}
	\le C \norm{R_hu_\varepsilon-u_{\varepsilon,h}}_{L^2(\Omega\backslash D)}.
\]
As $R_hu_\varepsilon-u_{\varepsilon,h}$ fulfills homogeneous boundary conditions on $\partial\Omega$, the estimate of the assertion is finally a consequence of the Poincar\'{e} inequality.
\end{proof}
For the remainder of this section let $D$ denote a strip at the boundary where we have
\begin{equation}\label{eq:betaassumption}
	\beta_\varepsilon(u_\varepsilon-\psi)|_D=\beta_\varepsilon(u_{\varepsilon,h}-\psi)|_D=0.
\end{equation}
This is  the same strip as  introduced in Theorem~\ref{thm:E:D:L2D_to_L2Dd} when we collect all the intermediate estimates in Theorem~\ref{thm:E:totalerror}. 
In a next step for the
final result, we estimate $u_\varepsilon-u_{\varepsilon,h}$ in $L^\infty(\Omega\backslash D)$. As already announced, we use a duality argument for that purpose. For the corresponding dual problem, we define
\begin{equation}\label{eq:b}
	b :=
	\begin{cases}
		[\beta_\varepsilon(\ue - \psi) - \beta_\varepsilon(\uh- \psi)]/(\ue-\uh) & \mbox{ if } (\ue-\uh)(x) \ne  0,\\
		0 &  \mbox{ else}.
	\end{cases}
\end{equation}
Note that $0\le b \le \varepsilon^{-1}$ almost everywhere in $\Omega$. Moreover, let 
\begin{equation}\label{eq:delta}
	\tilde \delta \text{ be a function from } C^\infty(\Omega) \text{ with } \operatorname{supp} {\tilde\delta}\subset \Omega\setminus D \text{ and } \lVert {\tilde\delta} \rVert_{L^1(\Omega)} \leq 1.
\end{equation}
Then, we define $G \in H^1_0(\Omega)$ as the weak solution to the dual problem
\begin{equation} \label{eq:E:defG_dual}
	\begin{split}
		-\Delta G +  bG & = {\tilde\delta}\quad \mbox{ in } \Omega,\\
		G &= 0 \quad \mbox{ on } \partial\Omega.
	\end{split}
\end{equation}
Before applying the duality argument in Theorem \ref{thm:E:I:ue_m_uh_dual}, let us state several auxiliary results.

\begin{lemma}\label{lem:bG} Let $D$ with $\abs{D}\ge0$ be a strip at the boundary where \eqref{eq:betaassumption} holds. Moreover, let $b$ and $\tilde\delta$ be the functions from \eqref{eq:b} and $\eqref{eq:delta}$, respectively, and let $G\in H^1_0(\Omega)$ be the solution of \eqref{eq:E:defG_dual}. Then, there holds
	\begin{enumerate}[label=(\roman*)]
	\item $\norm{bG}_{L^1(\Omega)}\le 1,$
	\item $\operatorname{supp} bG \subset \Omega\setminus D$.
	\end{enumerate}
\end{lemma}
\begin{proof}
(i) For $t>0$ we define the regularized sign function $\sgn_t(x) := \frac{x}{\sqrt{x^2+t}}$. 
Testing \eqref{eq:E:defG_dual} with $\sgn_t(G)$ yields
\begin{align*}
1 &\geq ({\tilde\delta},\sgn_t(G)) = (\nabla G, \sgn_t^\prime(G) \nabla G) + (bG, \sgn_t(G)).
\end{align*}
As a consequence, by means of the monotonicity of $\sgn_t$, we get $1\geq (bG, \sgn_t(G))_{L^2(\Omega)}$. 
Sending $t$ to zero and recalling that $b\geq 0$ yields 
$1 \geq \int_{\Omega} b\sgn(G)G = \lVert bG \rVert_{L^1(\Omega)}$. 
	
(ii) According to \eqref{eq:betaassumption} we have $\beta_\varepsilon(u_\varepsilon-\psi)=\beta_\varepsilon(u_{\varepsilon,h}-\psi)=0$ a.e. on $D$ such that $b=0$ a.e. on $D$, and hence $bG=0$ a.e. on $D$.
\end{proof}

\begin{lemma}
\label{lem:E:I:GmGh_L1_dual}
Let $D$ with $\abs{D}>0$ (independent of $\varepsilon$ and $h$) be a strip at the boundary where \eqref{eq:betaassumption} holds. Moreover, let $b$ and $\tilde\delta$ be the functions from \eqref{eq:b} and $\eqref{eq:delta}$, respectively. Then, there exists a constant $h_d>0$ such that for all $h\le h_d$ the solution $G$ of \eqref{eq:E:defG_dual} and its Ritz-projection $R_hG$ fulfill
\begin{align*}
  \|G-R_hG\|_{L^1(\Omega)} 
  \leq  C   h^2\abs{\log h}^2
\end{align*}
with a constant $C>0$ independent of $\varepsilon$, $h$ and $\tilde\delta$.
\end{lemma}
\begin{proof}
Let $z \in H^1_0(\Omega)$ denote the unique weak solution to
\begin{alignat*}{2}
  -\Delta z &= \sgn(G-R_hG) &&\quad \mbox{in } \Omega,\\
  z &= 0 &&\quad \mbox{on } \partial\Omega.
\end{alignat*}
By means of this equation, the orthogonality of the Ritz-projection, \eqref{eq:E:defG_dual} and Lemma~\ref{lem:bG}, we obtain
\begin{align*}
\norm{G-R_hG}_{L^1(\Omega)}
&=(G-R_hG,\sgn(G-R_hG))
=(\nabla (G-R_hG),\nabla z)\\
&=(\nabla (G-R_hG),\nabla(z-R_hz))
= (\nabla G, \nabla (z-R_hz))\\
&=({\tilde\delta}-bG, z-R_hz)_{L^2(\Omega\setminus D)}
\leq (\norm{{\tilde\delta}}_{L^1(\Omega)}+\norm{bG}_{L^1(\Omega)})\|z-R_hz\|_{L^\infty(\Omega\setminus D)}\\
&\le 2 \|z-R_hz\|_{L^\infty(\Omega\setminus D)}.
\end{align*}
For technical reasons, we have to introduce another subset $D'$ of $\Omega$, which is smoothly bounded, fulfills $D'\subset D$, and has a fixed and positive distance to $D$ and to $\partial \Omega$.
Using local $L^\infty$-error estimates from \cite[Theorem 10.1]{1991_HandbookOfNumAna__Wahlbin_LocalBehaviorFEM} in combination with a standard interpolation error estimate we get
for $h$ small enough
\begin{align*}
  \|z-R_hz\|_{L^\infty(\Omega\setminus D)}
  &\leq C \left(h^{2-\frac{N}{p}} \abs{\log h} \norm{z}_{W^{2,p}(\Omega\setminus D')}
  + \norm{z-R_hz}_{L^2(\Omega)}\right).
\end{align*}
A standard $L^2(\Omega)$-error estimate for the Ritz-projection together with elliptic regularity for $z$, and Lemma \ref{lem:L:local_W2p} implies

\begin{align*}
  \|z-R_hz\|_{L^\infty(\Omega\setminus D)}
  & \leq C\left( ph^{2-\frac{N}{p}}\abs{\log h} \norm{\sgn(G-R_hG)}_{L^p(\Omega)}
  + h^2 \norm{\sgn(G-R_hG)}_{L^2(\Omega)}\right)\\
  &\le C h^2 \abs{\log h}^2(p\abs{\log h}^{-1}h^{-\frac{N}{p}}+1),
\end{align*}
where we used that $\norm{\sgn(G-R_hG)}_{L^\infty(\Omega)}\le 1$. If we set $p=\abs{\log h}$, the desired result follows as $h^{-\frac{N}{\abs{\log h}}}=e^N$.
\end{proof}

\begin{theorem} \label{thm:E:I:ue_m_uh_dual}
Let $D$ with $\abs{D}>0$ (independent of $\varepsilon$ and $h$) be a strip at the boundary where \eqref{eq:betaassumption} holds. Moreover, let $u_\varepsilon$ and $u_{\varepsilon,h}$ be the solutions of \eqref{eq:R:PDE} and \eqref{eq:D:PDE}, respectively. Then, there exists a constant $h_d>0$ such that for all $h\le h_d$ there holds
\begin{equation*}
\|\ue-\uh\|_{L^\infty(\Omega\setminus D)} \leq C h^2 \abs{\log h}^2 ( \|f \|_{L^\infty(\Omega)} + \|\Delta\psi \|_{L^\infty(\Omega)} )
\end{equation*}
with a constant $C>0$ independent of $\varepsilon$ and $h$.
\end{theorem}

\begin{proof}
As $L^\infty(\Omega\setminus D)=(L^1(\Omega\setminus D))^*$ we have that
\[
\lVert \ue-\uh \rVert_{L^\infty(\Omega\setminus D)} 
= \sup_{\substack{{\tilde\delta}\in C^\infty(\Omega)\\ 
\operatorname{supp} {\tilde\delta} \subset \Omega\setminus D\\ 
\lVert {\tilde\delta} \rVert_{L^1(\Omega)} \leq 1}} 
\left|\int_{\Omega}(\ue-\uh) {\tilde\delta}\right|.
\]
Let such a ${\tilde\delta}$ be the right hand side of \eqref{eq:E:defG_dual}. 
Consequently, we get
\begin{align*}
\int_{\Omega}(\ue-\uh){\tilde\delta}
  &= (\nabla (\ue-\uh),\nabla G) + (\beta_\varepsilon(u_\varepsilon-\psi)-\beta_\varepsilon(u_{\varepsilon,h}-\psi),G) \\
  &= (\nabla (\ue-\uh),\nabla (G-R_hG) + (\beta_\varepsilon(u_\varepsilon-\psi)-\beta_\varepsilon(u_{\varepsilon,h}-\psi),G-R_hG),
\end{align*}
where we also used \eqref{eq:R:PDE} and \eqref{eq:D:PDE}. The orthogonality of the Ritz-projection and~\eqref{eq:R:PDE} imply
\begin{align}
\int_{\Omega}(\ue-\uh){\tilde\delta}& = (\nabla \ue,\nabla (G-R_hG) + (\beta_\varepsilon(u_\varepsilon-\psi)-\beta_\varepsilon(u_{\varepsilon,h}-\psi),G-R_hG)\nonumber\\
  & = (f - \beta_\varepsilon(\uh-\psi),G-R_hG)\nonumber\\
  &\le \norm{f - \beta_\varepsilon(\uh-\psi)}_{L^\infty(\Omega)}\norm{G-R_hG}_{L^1(\Omega)}. \label{eq:E:ue_m_uh_leq_data}
\end{align}
The assertion follows from the boundedness $\|\beta_\varepsilon(\uh-\psi)\|_{L^\infty(\Omega)} \leq c\|f+\Delta \psi\|_{L^\infty(\Omega)}$ according to \eqref{eq:R:betaEps}, and Lemma \ref{lem:E:I:GmGh_L1_dual}.
\end{proof}

If we now combine the results from Theorem~\ref{thm:R:LinfError}, Lemma~\ref{lem:E:P:InactiveAtBoundary}, Theorem~\ref{thm:E:D:L2D_to_L2Dd}, and Theorem~\ref{thm:E:I:ue_m_uh_dual}, 
as outlined in \eqref{eq:assembling}, we obtain the following result.
\begin{theorem} \label{thm:E:totalerror}
Let $u$, $u_\varepsilon$ and $u_{\varepsilon,h}$ be the solutions of \eqref{eq:C:VI}, \eqref{eq:R:PDE} and \eqref{eq:D:PDE}, respectively. In addition assume that $\psi<0$ on the boundary.
Then, there exist constants $\varepsilon_d>0$ and $h_d>0$ 
such that for all $\varepsilon\le \varepsilon_d$ and $h\le h_d$ there holds 
\begin{equation*}
  \|u-\uh\|_{L^2(\Omega)} \leq  C\left(\varepsilon + h^2\abs{\log h}^2\right) ( \|f \|_{L^\infty(\Omega)} + \|\Delta\psi \|_{L^\infty(\Omega)} )
\end{equation*}
with a constant $C>0$ independent of $\varepsilon$ and $h$, and using $\varepsilon = \mathcal O(h^2\abs{\log h}^2)$ we get
\begin{equation*}
  \|u-\uh\|_{L^2(\Omega)} \leq  Ch^2\abs{\log h}^2 ( \|f \|_{L^\infty(\Omega)} + \|\Delta\psi \|_{L^\infty(\Omega)} ).
\end{equation*}
\end{theorem}

We close this section with some remarks on certain additional aspects of our approach.

\begin{remark}[Inactivity at the boundary \texorpdfstring{$\partial\Omega$}{Lg}]
The previous results are derived under the assumption that the obstacle is inactive on the boundary. This is due to the appearance of singular terms within the primal and dual solutions
at the singular points of the boundary, which are the corners of the domain for $N=2$, and the corners and edges for $N=3$. However, the singularities are local phenomena. Away from the singular points, the regularity of the primal and dual solutions is only limited by the regularity of the data and the obstacle. For that reason, it is also sufficient to only assume inactivity of $\psi$ on the boundary at the singular points.
\end{remark}

\begin{remark}[Non-convex domains]
	Throughout the whole paper, we have assumed that the domain is convex. Let us briefly comment on the non-convex case. As already noticed in the previous remark, the singularities are only local phenomena around the singular points. Thus, the $W^{2,p}(\Omega\backslash D')$ regularity in the interior of the domain still holds. Only the $H^2(\Omega)$ regularity up to the boundary might no longer be true. Consequently, in the non-convex case, one only has to replace the estimates for
	\[
		\lVert u_\varepsilon-R_hu_\varepsilon\rVert_{L^2(\Omega)} \text{ in \eqref{eq:assembling} and }\lVert z-R_hz\rVert_{L^2(\Omega)} \text{ within the proof of Lemma \ref{lem:E:I:GmGh_L1_dual}}
	\]
	by the correspondingly adapted estimates. Moreover, as the results of Lemma \ref{lem:L:H1D_L2Dd} will also hold on sharply varying grids (see the discussion before Lemma \ref{lem:L:H1D_L2Dd}), it is also possible to use mesh grading techniques to retain the full order of convergence for the critical terms from above. 
\end{remark}

\section{Numerical Validation}
\label{sec:N}
For the numerical realization of the fully discrete equation~\eqref{eq:D:PDE}
we employ the finite element toolbox \texttt{Gascoigne} \cite{gascoigne}. 
To access  \texttt{Gascoigne} we use the optimization toolbox \texttt{RoDoBo} \cite{rodobo} as an interface. The resulting non-linear and non-smooth systems of equations corresponding to~\eqref{eq:D:PDE} 
are solved numerically by means of the semi-smooth Newton method (see e.g. \cite{Ulbrich2009})
provided by \texttt{Gascoigne}. In each of the numerical examples of the subsequent Sections~\ref{sec:N:sharpness} and~\ref{ssec:largestinnerangle} the discrete subspaces $V_h$ are constructed by piecewise bilinear and globally continuous functions on a sequence of subdivisions of $\Omega$ into
quadrilaterals. The computational domains $\Omega\subset \R^2$ are exactly specified below. Moreover, we fix $f=-30$ and $\psi = -1$. The constant $c$ in the definition of $\lambda$ \eqref{eq:lambda} is chosen as $c=6$. 

The example in Section~\ref{sec:N:sharpness} shows that the convergence rates of $u_{\varepsilon,h}$ in terms of $\varepsilon$ and $h$ are sharp. More precisely, we see that the exponents of $\varepsilon$ and $h\abs{\log h}$ in
\begin{align}
	\label{eq:N:errL2}
	\lVert u - \uh\rVert_{L^2(\Omega)} \leq C(\varepsilon + h^2|\log h|^2),
\end{align}
proven in Theorem~\ref{thm:E:totalerror}, can essentially not be improved.

The example in Section~\ref{ssec:largestinnerangle} studies the influence of the largest interior angle of a polygonal domain on the convergence rates in $L^2(\Omega)$ and $L^\infty(\Omega)$. It illustrates that the result of Theorem~\ref{thm:E:totalerror}, and hence estimate \eqref{eq:N:errL2}, is valid in general convex domains. However, the convergence rates in $L^\infty(\Omega)$ may be reduced depending on the largest interior angle due to the appearance of corner singularities. Let us denote by $\alpha\in [\pi/3,\pi)$ the largest interior angle of the domain. Then, one can show (neglecting $\log$-terms)
\begin{align}
	\label{eq:N:errLinf}
	\lVert u - \uh\rVert_{L^\infty(\Omega)} \leq C ( \varepsilon + h^{\min\{2,\pi/\alpha\}-\delta})
\end{align}
for an arbitrarily small $\delta>0$. For instance, this can be deduced from \cite[Lem. 2.2 and Thm. 2.3]{Nochetto_1988_Sharp_Linfty_semilinear_elliptic_Free_boundaries} 
having in mind the reduced regularity stemming from the corner singularities.

Before turning our attention to the numerical examples, we notice that we use reference solutions (computed on a fine mesh and with a small regularization parameter) for the purpose of comparison, 
as we do not have analytic solutions to any of our numerical examples.

\subsection{Validation of the Discretization Error Estimates in $L^2(\Omega)$} \label{sec:N:sharpness}
In this section we verify \eqref{eq:N:errL2}. As underlying domain we choose the unit square. 
The reference solution is computed with $\varepsilon_{ref} = 10^{-8}$ and $h_{ref} = 0.5^{10} \approx 10^{-3}$.
From the structure of the estimate one  expects that for small $h$ 
the total error is dominated by the error caused by 
$\varepsilon$ and vice versa.  
To show this, we calculate  solutions $\uh$ to  \eqref{eq:D:PDE}  for sequences $\varepsilon$ and $h$ tending to zero.
In Figure~\ref{fig:epstozerovarioush} we show $\lVert \uh-u_{ref} \rVert_{L^2(\Omega)}$ as a function of $\varepsilon$
for fixed values of $h$, while in Figure~\ref{fig:htozerovariouseps} we show $\lVert \uh-u_{ref} \rVert_{L^2(\Omega)}$ as a function of $h$ for fixed values of $\varepsilon$.

\begin{figure}
	\centering
	\input{imgs/epstozerodifferenth.tex}
	\caption{Sequence of errors for sequences of $h$ and $\varepsilon$ tending to zero.}
	\label{fig:epstozerovarioush}
\end{figure}

In Figure~\ref{fig:epstozerovarioush} we observe, that for every fixed $h$, the error becomes stationary for small $\varepsilon$ and cannot be further reduced by reducing $\varepsilon$. Hence, the discretization error is dominating in this case. 
Moreover, for $h$ sufficiently small we observe first order convergence in terms of $\varepsilon$. This is in agreement with our theoretical findings, see \eqref{eq:N:errL2}.
An analogous result is observed in Figure~\ref{fig:htozerovariouseps}, but with $\varepsilon$ and $h$ changing their roles. Of course, in terms of $h$ we see a convergence rate of close to two.

\begin{figure}
\centering
\input{imgs/htozerodifferenteps.tex}
\caption{Sequence of errors for sequences of $h$ and $\varepsilon$ tending to zero.}
\label{fig:htozerovariouseps}
\end{figure}

\subsection{Influence of the Largest Interior Angle on the Error in $L^2$ and $L^\infty$}
\label{ssec:largestinnerangle}

In this section we verify \eqref{eq:N:errL2} and \eqref{eq:N:errLinf} on domains $\Omega$ with varying largest interior angle.
The reference solution for each experiment is calculated with $\varepsilon_{ref} = 10^{-4}$ and $h_{ref} = 0.5^9 \approx 2\cdot 10^{-3}$.
Moreover we fix $\varepsilon = 10^{-4}$ for all experiments, and hence, we only investigate the error behavior with respect to $h$ depending on the largest interior angle.
As computational domains, we consider the domains $\Omega_\alpha$ with largest interior angle $\alpha\in[\pi/2,\pi)$ which are defined by
\begin{align*}
\bar\Omega_\alpha := \operatorname{conv} \{ (0,0), (1,0), (0,1), ( 1 + \tan( \alpha / 2 ) )^{-1} (1,1) \}.
\end{align*}
In particular, the case $\alpha = \frac{\pi}{2}$ leads to the unit square $(0,1)^2$, while
for $\alpha\rightarrow\pi$ the domain $\Omega_\alpha$ degenerates to a rectangular triangle.

We perform experiments for three particular domains with largest interior angle $\pi/2$, $3\pi/4$, and $17\pi/18$.
Our observations are presented in 
Tables~\ref{fig:largestangleinfluenceDomain1} -- \ref{fig:largestangleinfluenceDomain3}. 
Here $\eta_h^{L^p} := \lVert \uh -u _{ref}\rVert_{L^p(\Omega)}$ 
abbreviates the error between the numerical solution $\uh$ and the reference solution $u_{ref}$
in the $L^p$-norm ($p \in \{2,\infty\}$). 
For sequences $(h_k)$ and $(\eta_k)$ we define the experimental order of convergence (EOC)
by 
\begin{align*}
\mathrm{EOC}_k = \frac{\log(\eta_k) - \log(\eta_{k-1})}{\log(h_k) - \log(h_{k-1})}
\end{align*}
as an approximation to the convergence rate of $(\eta_k)$ with respect to $(h_k)$.
We observe that the experimental orders of convergence for $\eta_h^{L^2}$ are two on all three domains, 
as expected from \eqref{eq:N:errL2}.
In case of $\eta_h^{L^\infty}$, we observe a decreasing convergence rate for an increasing largest interior angle. The corresponding experimental orders of convergence nicely follow the theoretical result from \eqref{eq:N:errLinf}.

\begin{table}
\centering
\begin{tabular}{ccccc}
\hline
$h$ & $\eta^{L^2}_h$ & EOC of $\eta_h^{L^2}$ & $ \eta^{L^\infty}_h$ &  EOC of  $ \eta_h^{L^\infty}$ \\
\hline
$0.5^4$ & $6.17 \cdot 10^{-3}$ & $-$    & $9.47 \cdot 10^{-3}$  & $-$ \\
$0.5^5$ & $1.50 \cdot 10^{-3}$ & $2.04$ & $2.39 \cdot 10^{-3}$  & $1.99$ \\
$0.5^6$ & $3.76 \cdot 10^{-4}$ & $2.00$ & $5.96 \cdot 10^{-4}$  & $2.00$ \\
$0.5^7$ & $9.11 \cdot 10^{-5}$ & $2.05$ & $1.43 \cdot 10^{-4}$  & $2.06$ \\
\hline
\multicolumn{2}{l}{expected} & $2$ & & $2-\delta$ \\
\hline
\end{tabular}
\caption{Numerical results for $\Omega_{\frac{1}{2}\pi}$.}
\label{fig:largestangleinfluenceDomain1}
\end{table}

\begin{table}
\centering
\begin{tabular}{ccccc}
\hline
$h$ & $\eta^{L^2}_h$ & EOC of $\eta_h^{L^2}$ & $\eta^{L^\infty}_h$  & EOC of $\eta_h^{L^\infty}$ \\
\hline
$0.5^4$ & $3.91 \cdot 10^{-3}$ & $-$    & $1.17 \cdot 10^{-2}$ & $-$ \\
$0.5^5$ & $9.37 \cdot 10^{-4}$ & $2.06$ & $4.61 \cdot 10^{-3}$ & $1.34$ \\
$0.5^6$ & $2.29 \cdot 10^{-4}$ & $2.03$ & $1.81 \cdot 10^{-3}$ & $1.35$ \\
$0.5^7$ & $5.51 \cdot 10^{-5}$ & $2.06$ & $6.85 \cdot 10^{-4}$ & $1.40$ \\
\hline
\multicolumn{2}{l}{expected} & $2$  & & $\frac{4}{3}-\delta$ \\
\hline
\end{tabular}
\caption{Numerical results for $\Omega_{\frac{3}{4}\pi}$.}
\label{fig:largestangleinfluenceDomain2}
\end{table}

\begin{table}
\centering
\begin{tabular}{ccccc}
\hline
$h$ & $\eta^{L^2}_h$ & EOC of $\eta_h^{L^2}$ & $\eta^{L^\infty}_h$ &  EOC of $\eta_h^{L^\infty}$ \\
\hline
$0.5^4$ & $2.48 \cdot 10^{-3}$ & $-$    & $8.64 \cdot 10^{-3}$ & $-$ \\
$0.5^5$ & $6.41 \cdot 10^{-4}$ & $1.95$ & $3.65 \cdot 10^{-3}$ & $1.24$ \\
$0.5^6$ & $1.62 \cdot 10^{-4}$ & $1.98$ & $2.00 \cdot 10^{-3}$ & $0.87$ \\
$0.5^7$ & $3.94 \cdot 10^{-5}$ & $2.04$ & $9.22 \cdot 10^{-4}$ & $1.11$ \\
\hline
\multicolumn{2}{l}{expected} & $2$ & & $\frac{18}{17}-\delta$ \\
\hline 
\end{tabular}
\caption{Numerical results for $\Omega_{\frac{17}{18}\pi}$.}
\label{fig:largestangleinfluenceDomain3}
\end{table}

\bibliographystyle{apalike}
\bibliography{HafKPV_NumAna_Obstacle}

\end{document}